\documentclass{birkjour}
\usepackage{epsfig,enumerate,amsmath,amsfonts,amssymb}
\usepackage{indentfirst}
\usepackage{pstricks}
\usepackage{latexsym}
\numberwithin{equation}{section}
\def\dis{\displaystyle}

\def\ato0{{\buildrel{\dis\longrightarrow}\over{a\to0}}}

\def\Re{\mathbb{R}}
\newcommand{\ud}{\mathrm{d}}

\newtheorem{thm}{Theorem}[section]

\newtheorem{lem}[thm]{Lemma}
\newtheorem{prop}[thm]{Proposition}
\newtheorem{defi}{Definition}[section]
\newtheorem{thmlet}{Theorem}

\newtheorem{remark}{Remark}[section]

\newcommand{\R}{{\mathbb R}}

\numberwithin{equation}{section}

\begin{document}
\title[Fractional Hamiltonian systems with critical exponential growth]{Nonautonomous fractional Hamiltonian\\ system with critical exponential growth}
\author[Joao Marcos do \'O]{Jo\~ao Marcos do \'O}
\email{jmbo@pq.cnpq.br}
\address[Joao Marcos do \'O]{Department of Mathematics,
Bras\'{\i}lia University, 70910-900, Bras\'ilia, DF, Brazil}

\author[Jacques Giacomoni]{Jacques Giacomoni}
\email{jacques.giacomoni@univ-pau.fr}
\address[Jacques Giacomoni]{LMAP (UMR {E2S-UPPA} CNRS 5142), Bat. IPRA, Avenue de l'Universit\'e, F-64013 Pau, France}

\author[Pawan Kumar Mishra]{Pawan Kumar Mishra}
\email{pawanmishra@mat.ufpb.br}

\address[Pawan Kumar Mishra]{Department of Mathematics, Federal University of Para\'{\i}ba, 58051-900, Jo\~ao Pessoa-PB, Brazil}



\date{November 3, 2018}
\begin{abstract}
In this paper, we  study the following nonlocal nonautonomous Hamiltonian system on whole $\mathbb R$ 
$$
\left\{\begin{array}{ll}
(-\Delta)^\frac12~ u +u=Q(x) g(v)&\quad\mbox{in } \mathbb R,\\
(-\Delta)^\frac12~ v+v = P(x)f(u)&\quad\mbox{in } \mathbb R,
\end{array}\right.
$$
where $(-\Delta)^\frac12$ is {the} square root  Laplacian operator. We assume that the nonlinearities $f, g$ have critical growth at $+\infty$ in the sense of Trudinger-Moser inequality and the nonnegative weights $P(x)$ and $Q(x)$ vanish at $+\infty$. Using suitable variational method combined with {the} generalized linking theorem, we obtain the existence of {at least one} positive solution for the above system.
\end{abstract}
\maketitle
\section{  Introduction and main results}
In this paper, we study the following system
\begin{equation}\label{P}
\left\{\begin{array}{ll}
(-\Delta)^\frac12~ u +u=Q(x) g(v)&\quad\mbox{in } \mathbb R,\\
(-\Delta)^\frac12~ v+v = P(x)f(u)&\quad\mbox{in } \mathbb R,
\end{array}\right.
\end{equation}
 where  $(-\Delta)^\frac12$ is the square root  Laplacian  operator defined as
\[
(-\Delta)^\frac12 u(x)=-\frac{1}{2\pi} \int_{-1}^1 \frac{u(x+y)+u(x-y)-2u(x)}{|y|^2}dy.
\] 
One difficulty in studying Hamiltonian elliptic systems via variational methods is that the energy
functional is strongly indefinite, that is, its quadratic part is respectively coercive and anti-coercive
in infinite dimensional subspaces of the energy space. To deal with such difficulty we use a Galerkin method, introduced by {Rabinowitz} in \cite{MR0301587, MR0467823}.
Another important obstacle is to handle the lack of compactness, which roughly speaking,
originates from the non-compactness of the Trudinger-Moser embedding. To add on, {we face a} lack of compactness of Sobolev embedding {because} the problem is posed in whole $\mathbb R$.

In the local case {{\it i.e.}} in the case of the standard Laplacian operator, the existence of solution for Hamiltonian elliptic systems has been extensively studied in the
literature mostly in higher dimensions involving Sobolev critical growth in  bounded as well as unbounded domain of $\mathbb R^N$. For the case of a
bounded domain, see for instance, \cite{MR537061, MR1177298, MR2146049, MR1214781, MR1220982, MR1422612}. On the other hand, Hamiltonian elliptic systems in whole $\mathbb R^N$ have been explored to a lesser extend, see for example, \cite{MR1617988, MR2059153, MR1785681, MR2661494}. We refer the reader to \cite{MR3298466} for
a recent survey on this subject.

For $N=2$, Hamiltonian systems in a bounded domain in $\mathbb R^2$ have been studied by D. G. de Figueiredo, J. M. do \'O and B. Ruf \cite{MR2095447} in the critical growth range.
We cite N. Lam and G. Lu  \cite{MR3145918} for a similar result without Ambrosetti-Rabinowitz condition (see $(H3)$). In whole $\mathbb R^2$, D.G.  de Figueiredo, J. M. do \'O and J. Zhang \cite{DDJ} studied the ground state solution for the system \eqref{P} with $P=Q=1$  using the idea of generalized Nehari manifold under a monotonicity assumption on nonlinearity {{\it i.e.}} $f(t)/|t|$ and $g(t)/|t|$ are strictly increasing in $(-\infty, 0)$ and $(0, +\infty)$. 
\noindent Another reference in unbounded domain is of  M. de Souza and J. M. do \'O  \cite{MR3500314}, where  authors have considered a system of the type
 \begin{equation*}
\left\{\begin{array}{ll}
-\Delta u +V(x)u= g(v)&\quad\mbox{in } \mathbb R^2,\\
-\Delta v +V(x) v= f(u)&\quad\mbox{in } \mathbb R^2.
\end{array}\right.
\end{equation*}
Under some suitable conditions on $V$, the loss of compactness in critical case was recovered.

We know from classical fractional Sobolev embedding that $H^{s,2}(\mathbb R^N)$ is continuously embedded in $L^q(\mathbb R^N)$ for all $q\in [2, 2^*_s]$, where $2^*_s=2N/(N-2s)$. Note that formally, $2^*_s=\infty$ if $N=2s$. The only choice for this fact to be true is $N=1$ and $s=1/2$, since $s\in (0, 1).$  At this point a natural question arises: What is the optimal space where $H^{1/2,2}(\mathbb R)$ is embedded?  This answer was first given by Ozawa \cite{MR1317718} and later improved by  Iula,  Maalaoui and  Martinazzi \cite{MR3392369} which we have stated in Theorem \ref{Prop:fractional TM} below.

In case of $u=v$, $f=g$ and $P=Q$, the system \eqref{P} converts into the following scalar equation 
\begin{equation}\label{scalar}
(-\Delta)^\frac12 u+u=P(x)f(u)\;\;\textrm{in}\;\; \mathbb R.
\end{equation}

Motivated from fractional Trudinger Moser inequality as in Theorem \ref{Prop:fractional TM}, J. M. do \'O, Miyagaki and Squassina \cite{MR3399183} studied the existence of postive solutions to \eqref{scalar} with a class of weights $P$ containing the Lebesgue integrable functions and nonlinearities $f$ having subcritical and critical exponential growth (see also \cite{DEA}). In the case of an open and bounded interval we cite an earlier work of Iannizzotto and Squassina   \cite{IzSq}, where authors have proved  existence and multiplicity of positive solutions with critical exponential growth (see \cite{JPS} also). We cite \cite{WaHc, VaMbDr, XcJt, Tan, XYu} for more results in the case of Sobolev critical growth for dimension $N\geq 2$ involving the square root of Laplacian.

Inspired from the above literature, we consider in the present paper an Hamiltonian system involving the square root Laplacian operator posed in whole space ${\mathbb R}$ and investigate the effect of critical exponential growth of nonlinearities. We use a generalized linking theorem as compared to other results in the literature related to Hamiltonian systems which are heavily dependent on finite dimension Galerkin approximation to prove existence of at least one positive solution. For that, facing of the lack of compactness due to the critical growth of nonlinearities and unboundedness of the domain, we need to investigate accurately the behavior of suitable Palais Smale sequences. As far as we know there is no result regarding such class of fractional Hamiltonian systems. Our  results appear to be the first of its kind for fractional Hamiltonian systems and we expect that the insights and the methods used in the present paper apply to a wider class of semilinear elliptic operators. 
\subsection{ Critical exponential nonlinearity}
It is usual in the literature to say that {$ h $}
has critical growth of Trudinger-Moser type at $ + \infty $ if there exists $
\alpha_0 >0 $ such that
\[
    \lim_{s \rightarrow + \infty} \frac{h(s)}{e^{\alpha s^2}-1} =
    \left\{
      \begin{array}{cl}
        0 & \quad\hbox{if}\quad \alpha > \alpha_0, \\
        +\infty & \quad\hbox{if}\quad \alpha < \alpha_0 .
      \end{array}
    \right.
   \]
We note that such notion is motivated by a fractional version of Trudinger-Moser
inequality in the whole space $\mathbb{R}$ as follows.
\begin{thmlet}{(A fractional Trudinger-Moser inequality)}
\label{Prop:fractional TM}
It holds
\begin{align*}
	\sup_{u \in {H}^{1/2,2}(\mathbb R),\; \| u \|_{1/2} \le 1} \int_{\mathbb R} (e^{\alpha |u|^2}-1)~ dx
	\begin{cases}
	&< \infty, \quad \alpha \le \pi, \\
	&= \infty, \quad \alpha > \pi,
	\end{cases}
\end{align*}
\end{thmlet}
where  ${H}^{1/2,2}(\mathbb R)$ is the fractional order Sobolev space equipped with $\|\cdot \|_{1/2}$ norm which is defined in Section 2.
{
\subsection{Assumptions on the weights $P(x)$ and $Q(x)$}
We assume $P, Q\not\equiv\,0$ belonging to $ C_0(\mathbb R, \mathbb R^+)$ that is,   $P, Q$ are continuous, nonnegative and
\begin{equation}\label{vanishing}
\lim_{|x|\to +\infty} P(x)=0=\lim_{|x|\to +\infty} Q(x).
\end{equation} 
Without loss of generality,  for the ease of reference, we assume: 
\begin{equation}\label{LDT}
{ \sqrt{P(0)Q(0)}=1}
\end{equation}
\begin{remark}
We recall that the assumption \eqref{vanishing} does not imply that $P(x)$ and $Q(x)$ are {Lebesgue} integrable. For example take 
$P(x)=1$ if $|x|\leq 1$ otherwise $ {1}/{|x|}$.
\end{remark} 
}
\subsection{Assumptions on the nonlinearities $f$ and $g$}
 We consider the following assumptions on nonlinearities $f$ and $g$ which are $\alpha_0-$ critical at $+\infty$.
\begin{itemize}
\item [$(H1)$] (continuity) $ f, g:\mathbb R\rightarrow [0,  \infty)$ are continuous functions  {and $f(0)=g(0)=0$ }.
\item [$(H2)$] (behavior near the origin)\;$f(t)=o(t)$ and $g(t)=o(t)$ near the origin.
\item [$(H3)$] (Ambrosetti-Rabinowitz type condition) there exist constants $\theta>2$ such that , for all $t>0$, one has
$$
0<\theta F(t):=\theta \int_0^t  f(s)\,\ud s  \leq t f(t)\;\;\text{and}\;\; 0<\theta G(t):=\theta\int_0^t  g(s)\,\ud s\leq t g(t).$$
\item [$(H4)$] (asymptotic behavior) $\lim_{t\rightarrow +\infty}t f(t)e^{-\alpha_0 t^2}=+\infty$ and \\$\lim_{t\rightarrow +\infty}t g(t)e^{-\alpha_0 t^2}=+\infty.$
\end{itemize}

\subsection{Main result}
Before stating the main result of the paper, we define the following class of functions. We say that  a function $f$ belongs to the class $\mathbb C_{\mathrm P}$, if for any sequence $\{u_n\}\subset H^{1/2, 2}(\mathbb R)$ satisfying 
 for some positive constant $C$ \[ 
u_n\rightharpoonup 0\;\;\textrm{ and }\;\;\int_\mathbb R P(x)f(u_n)u_n \,\ud x<C \;\;\textrm{implies}\;\;\int_\mathbb R P(x) F(u_n)\,\ud x \to 0.
\]
\begin{thm} \label{thm1}
Assume $f$ and $g$ belong to the class $\mathbb C_{\mathrm P}$ and $\mathbb C_{\mathrm Q}$ respectively and have $\alpha_0-$critical growth at $+\infty$. Let $f$ and $g$  satisfy $(H1)-(H4)$ and $P$, $Q$ satisfy \eqref{vanishing} then the system (\ref{P}) possesses a positive solution.
\end{thm}
\begin{remark}\label{Val}
The assumption $(H5)$ given below together with $(H2)$ implies  that $f\in \mathbb C_{\mathrm P}$ and $ g\in \mathbb C_\mathrm{Q}$.
\begin{itemize}
\item[$(H5)$]  $\displaystyle\lim_{t\to \infty} \frac{F(t)}{tf(t)}=0$ and $\displaystyle\lim_{t\to \infty} \frac{G(t)}{tg(t)}=0.$
\end{itemize}
We highlight that when $f$ and $g$ are of $O(t^2)$ near origin, $\mathbb C_{\mathrm P}$ implies the assumption $(L6)$ adopted in \cite[Theorem 4]{MR3145918} in the critical case.
 However if $f$ and $g$ decay slowly at $0$, say are of $O(t^{1+\ell})$ with $\ell<1$, $u_n\rightharpoonup 0$ and $\{\int_\mathbb R P(x)f(u_n)u_n dx\}$ is bounded do not imply that $\int_\mathbb R P(x)f(u_n)dx\to 0$. Precisely, vanishing behavior may occur for the sequence $\{P(x)f(u_n)\}$ in this case (see Remark \ref{Vanishing}).
 To validate our remark, we have added a proof in the appendix of this paper.
 \end{remark}


{
\begin{remark}
 The result of Theorem \ref{thm1} holds good even for the class of positive weights satisfying 
  $P\in L^\gamma(\mathbb R) $ and $Q\in L^{\gamma'}(\mathbb R)$ with $\gamma$, and $\gamma'>1$ or $P,Q\in L^1(\mathbb R)\cap C(\mathbb R, \mathbb R)$.
\end{remark}
}
\begin{remark}
We have the following observations regarding the assumptions above.
\begin{itemize} 
\item [$(i)$]  The assumption $(H3)$ is a global Ambrosetti-Rabinowtiz condition which is used to prove the boundedness of the Palais-Smale sequence. 
\item [$(ii)$] Instead of $(H4)$ one can take even a slightly weaker assumption  
$(H4)':$\;
$\lim_{t\rightarrow +\infty}t f(t)e^{-\alpha_0 t^2}\geq \eta$ and $\lim_{t\rightarrow +\infty}t g(t)e^{-\alpha_0 t^2}\geq \eta$ where $\eta$ is a sufficiently large positive real number.
\item [$(iii)$] {Examples of functions satisfying $(H1)$-$(H5)$ are $f(t)=t^{\ell_1} e^{t^{\beta_1}}e^{\alpha_0 t^2}$,  $g(t)=t^{\ell_2} e^{t^{\beta_2}} e^{\alpha_0 t^2}$with $\ell_1,\ell_2>1$ and $0\leq\beta_1,\beta_2<2$.}
\end{itemize}
\end{remark}
{
\begin{remark}
We can consider different critical growth for $f$ and $g$, say $\alpha_0-$ and $\beta_0-$ critical growth respectively. Precisely, we can 
   use the scaling $\bar v=\sqrt{\frac{\beta_0}{\alpha_0}}v$ in \eqref{P} to get
\begin{equation*}
\left\{\begin{array}{ll}
(-\Delta)^\frac12~  u + u= Q(x)\bar g(\bar v)&\quad\mbox{in } \mathbb R,\\
(-\Delta)^\frac12~\bar v+\bar v = \overline{P}(x) f(u)&\quad\mbox{in } \mathbb R,
\end{array}\right.
\end{equation*}
where $\overline{P}(x)=\sqrt{\frac{\beta_0}{\alpha_0}}P(x)$ and $\bar g(\bar v)=g\left(\sqrt{\frac{\alpha_0}{\beta_0}} v\right)$. It is easy to see that $\bar g$ has the $\alpha_0-$ critical growth as $g$ and $\overline{P}\in C_0(\mathbb R, \mathbb R)$.
\end{remark}
}

The paper is organized as follows. In section 2, we discuss the abstract framework related to the problem \eqref{P}. A generalized version of Linking geometry and related estimates are shown in section 3. In particular, in Proposition \ref{progeometry}, we give an upper estimate of the energy level that guarantees the compactness of suitable Palais Smale sequences. Here the assumption $(H4)$ plays a crucial role. This estimate is subsequently used in Section 4 where the behavior of Palais Smale sequences is described. Finally, Section 5 contains the proof of the main result of Theorem \ref{thm1}.

\section{Abstract framework}
In order to apply variational methods, we recall { the Bessel potential space} $H^{s,p}(\mathbb R^N)$  as
\begin{align*}
	H^{s,p}(\mathbb R^N) &= \left\{ u \in L^p(\mathbb R^N) \, : \, (-\Delta)^{s/2} u \in L^p(\mathbb R^N) \right\}.
\end{align*}
On the other hand, the Sobolev-Slobodeckij space $W^{s,p}(\mathbb R^N)$ is defined as
\begin{align*}
	W^{s,p}(\mathbb R^N) &= \left\{ u \in L^p(\mathbb R^N) \, : \, [u]_{s, p, \mathbb R^N} < \infty \right\},
	\end{align*}
	where 
	\begin{align*}
	[u]_{s, p,\mathbb R^N}^p &= \int_{\mathbb R^N} \int_{\mathbb R^N} \frac{|u(x)-u(y)|^p}{|x-y|^{N + sp}} \,\ud x\ud y.
\end{align*}
In general, $H^{s,p}(\mathbb R^N) \ne W^{s,p}(\mathbb R^N)$ for $p \ne 2$.
In this paper, we are interested in the limiting Sobolev embedding case {{\it i.e.}} when $2^*_s=2N/(N-2s)=\infty$ which, as observed before, corresponds to $s = 1/2$ and $p = 2$ in dimension $N=1$.
In this case, the  space $H^{1/2,2}(\mathbb R)$ coincides with the Sobolev-Slobodeckij space $W^{1/2,2}(\mathbb R)$ and
both seminorms are related as
\begin{align*}
	 \| (-\Delta)^{1/4} u \|_{L^2(\mathbb R)}^2 = \frac{1}{2\pi} [u]_{W^{1/2,2}(\mathbb R)}^2,
\end{align*}
see Proposition 3.6. in \cite{MR2944369}.
\noindent The space $H^{{1}/{2}, 2}(\mathbb R)$ is the Hilbert space with the norm defined as
\begin{equation*}
 {\|u\|_{1/2}^2=\|u\|_{L^2({\mathbb R})}^2}+\int_{\mathbb R}|(-\Delta)^{\frac{1}{4}}u|^2\,\ud x
\end{equation*}
induced from the inner product given as 
\[
\langle u, v\rangle_{1/2}=\int_{\mathbb R}(-\Delta)^\frac14 u (-\Delta)^\frac14 v \,\ud x+\int_{\mathbb R} u v \,\ud x.
\]
{We recall that $H^{1/2,2}(\mathbb R)$ is continuously embedded in $L^q({\mathbb R})$, for any $q\in [2,\infty)$.} 

Consider the following weighted Banach space
\[
L^r(\mathbb R;P)=\left\{u:\mathbb R\rightarrow \mathbb R\;\textrm{measurable:}\; \displaystyle\int_{\mathbb R}P(x)|u|^r \,\ud x<\infty\right\}
\]
and similarly we define
\[
L^r(\mathbb R;Q)=\left\{u:\mathbb R\rightarrow \mathbb R\;\textrm{measurable:}\; \displaystyle\int_{\mathbb R}Q(x)|u|^r \,\ud x<\infty\right\}.
\] 
We state our first result related to the compactness of the embedding of $H^{1/2, 2}({\mathbb R})$ into weighted integrable spaces defined above as follows.
\begin{lem}\label{compact-a}
The space $H^{1/2, 2}(\mathbb R)$ is compactly embedded in $L^q(\mathbb R;{P})$ and $L^q(\mathbb R;Q)$ respectively for $q\in [2, \infty).$
\end{lem}

{
\begin{proof}
We prove only the compact embedding of   $H^{1/2, 2}(\mathbb R)$ into $L^q(\mathbb R;{P})$. The other one is similar.  Let $q\in [2,\infty)$ and $\epsilon >0$. Then from the assumption \eqref{vanishing}, there exists $L=L(\epsilon)>0$ large enough, such that 
\begin{equation}\label{consvan}
P(x)<\epsilon, \textrm{ for all }\; |x|>L.
\end{equation}
If $\{u_n\}\subset H^{1/2, 2}(\mathbb R)$ is such that $u_n \rightharpoonup u$ weakly in $ H^{1/2,2}(\mathbb R)$ for some $u\in  H^{1/2, 2}(\mathbb R)$,  then using the continuous injection of $H^{1/2,2}(\mathbb R)$ in an
arbitrary $L^r(\mathbb R)$ space with $r\in [2,\infty)$ there exist $M, M_r>0$ such that 
\begin{equation}\label{nthch}
\int_{\mathbb R} |(-\Delta)^\frac14u_n|^2 \,\ud x+ \int_{\mathbb R}|u_n|^2 \,\ud x \leq  M, 
  \int_{\mathbb R}|u_n|^{r} \, \ud x\leq M_r.
\end{equation}
Now from \eqref{consvan} and \eqref{nthch}, we get 
\begin{align*}
\int_{\mathbb R\setminus [-L, L]} P(x)|u_n|^q \,\ud x&\leq \epsilon \int_{\mathbb R\setminus [-L, L]}|u_n|^q \,\ud x \\& \leq \epsilon\int_{\mathbb R}|u_n|^q \,\ud x\leq \epsilon\|u_n\|_{1/2}^q\leq \epsilon M^{q/2}.
\end{align*}
Using in addition using the compact injection of $H^{1/2,2}(\mathbb R)$ in an arbitrary $L^r_{\mathrm loc}(\mathbb R)$ space with $r\geq 2$ together with the boundedness of $P(x)$ in the compact subset of $\mathbb R$ as $P(x)$ is continuous, we have compactness of the embedding of $H^{1/2, 2}(\mathbb R)$ in $L^q(\mathbb R;{P})$ for all $q \in [2, \infty)$.
\end{proof}
}
Next, associated with our system we consider the Hilbert space $\mathbb{H}= H^{1/2,2}(\mathbb R) \times H^{1/2,2}(\mathbb R)  $ with the inner product and norm
\[
\langle (u,v),(\varphi,\psi)\rangle:=\langle u,\varphi\rangle_{1/2}+\langle v,\psi \rangle_{1/2}, \; \; 
\|(u,v)\|:=\left(\|u\|_{1/2}^2+\|v\|_{1/2}^2 \right)^{1/2}.
\]

Consider the  natural associated functional with system \eqref{P}, defined in $\mathbb{H}$ by 
\begin{equation}\label{Functional}
    I(u,v):= \Phi(u,v) - \int_{\mathbb{R}} P(x) F(u)\,\ud x-\int_{\mathbb{R}} Q(x) G(v) \, \ud x, 
\end{equation}
where the associated quadratic part is defined by
\begin{equation*}
    \Phi(u,v):=\int_{\mathbb{R}} \left( (-\Delta)^{\frac14} u (-\Delta)^{\frac14} v +  u v \right) \,\ud x.
\end{equation*}
Using standard arguments it is possible to verify that $I$ is well defined
and is of class $C^1$ with
\begin{eqnarray*}
&I'(u,v)(\phi,\psi) = \int_{\mathbb{R}} ( (-\Delta)^{\frac14} u (-\Delta)^{\frac14} \psi +
(-\Delta)^{\frac14} v (-\Delta)^{\frac14} \phi + u \psi + v \phi ) \, \ud x \\
&-
\int_{\mathbb{R}} \left(P(x)f(u) \phi + Q(x)g(v) \psi\right) \, \ud x \quad\forall(\phi,\psi) \in \mathbb{H}.
\end{eqnarray*}
Consequently, critical points of the functional $ I $ are precisely the
weak solutions to (\ref{P}).

Note that the functional $I$ is strongly indefinite, since  $\Phi(u,v)>0$ when $0\not\equiv (u,v)\in$ $\mathbb{H}^+$ and $\Phi(u, v)<0$ whenever $0\not\equiv(u,v)\in\mathbb{H}^-$, where $
\mathbb{H}^{+} := \left\{(u,u) : u \in H^{1/2, 2}(\mathbb R) \right\},
$
  $
\mathbb{H}^{-} := \left\{(u,-u) : u \in  H^{1/2, 2}(\mathbb R) \right\}$ are infinite dimensional subspaces of ${\mathbb H}.
$

In order to deal with the fact that the functional $I$ is strongly indefinite, we shall use a version of the Palais-Smale condition inspired  by the Galerkin method which we describe next.
\subsection{Palais-Smale condition and Generalized Mountain Pass Theorem}
Let $\mathbb{W}$ be a real separable Banach space, and suppose $\{ \phi_i : i\in J \subset \mathbb{N}  \}$ is a basis of $\mathbb{W}$. Given a family $\{ J_n \}_{n\in \mathbb{N}}$ such that $J_1 \subset J_2 \subset \cdots  J_n \subset \cdots J  $, and $\cup_{n\in \mathbb{N}} J_n = J$, we set for every $n\in \mathbb{N}$,
\[
\mathbb{W}_n =  \overline{ \mathrm{span}\{\phi_i : i \in J_n\}} \quad \mbox{and} \quad I_n= I|_{\mathbb{W}_n}
\]

Now we shall consider the $(PS)^*$ condition with respect to the family $(\mathbb{W}_n)$ of $\mathbb{W}$ in the following sense:
\begin{defi} Given $c\in \mathbb{R}$, we say that $(u_n) \subset \mathbb{W}$ is a $(PS)_c^*$ sequence for the functional  $I\in C^1(\mathbb{W}, \mathbb{R})$ if 
\begin{description}
\item[(i)] There exists a sequence $(n_j) \subset \mathbb{N}, \; n_j\rightarrow \infty $ as $j\rightarrow\infty$ such that $u_{n_j}\in \mathbb{W}_{n_j}$, for every $j\in \mathbb{N}$;
\item[(ii)] $I(u_{n_j})\rightarrow c$  as $j\rightarrow\infty$;
\item[(iii)] $\| I_{n_j}'(u_{n_j}) \|_{\mathbb{W}^*_{n_j}}\rightarrow 0$, as $j\rightarrow\infty$.
\end{description}
\end{defi}

\begin{defi}
	Given $c \in \mathbb{R}$, we say that the functional $I \in C^1(\mathbb{W}, \mathbb{R})$ satisfies the $(PS)^*_c$ 
	condition (with respect to the family $(\mathbb{W}_n)$) if every $(PS)_c^*$ sequence for the functional  $I$
	possesses a subsequence converging to  a critical point of $I$.
\end{defi}

We apply the following version of Generalized Mountain Pass Theorem ( see \cite{MR845785}  and \cite{MR1093380}). Considering $\mathbb{W} = X_1 \oplus X_2$, we shall suppose  $I \in C^1(\mathbb{W}, \mathbb{R})$ satisfies the $(PS)^*_c$ 
condition with respect to the family  $(\mathbb{W}_n)\subset \mathbb{W}$ with $\mathbb{W}_n = X^n_1 \oplus X^n_2, \; X_i^n \subset X_i, \; i=1,2$, $\dim(X_1^n)<\infty$ for any $n\in \mathbb{N}$.

\begin{thm} \label{linking}(Generalized Mountain Pass Theorem)
Let $\mathbb{W}= X_1 \oplus X_2$ be a real Banach space. 
Suppose $I \in C^1(\mathbb{W}, \mathbb{R})$ satisfies
\begin{description}
\item[$(I_0)$]
$\text{For every} \quad u \in X_1$, it holds $ I(u) \leq 0$.
\item[$(I_1)$]
There exist $\rho > 0$ and $\sigma>0$ such that
$$
I(u) \geq \sigma, \quad \text{for every} \quad u \in \partial B_{\rho}(0) \cap X_2.
$$
\item[$(I_2)$]
For each $n\in \mathbb{N}$, there exist $e_n \in \partial B_{1}(0) \cap X^n_2$ and $\beta \in \mathbb{R}$ such that
$$
I(u) \leq \beta, \quad \text{for every}\quad u \in X_1 \oplus \mathbb{R}^+ e_n,
$$
where
$$
X_1 \oplus \mathbb{R}^+ e_n=\left \{u=v+te_n \in X_1 \oplus \mathbb{R}e_n : v \in X_1, t \geq 0 \right \}.
$$
\end{description}
If $I$ satisfies $(PS)_c^*$ for every $c \in [\sigma,\beta]$, then $I$ possesses a critical point  $u\in \mathbb{W}$ such that $I(u)\in [\sigma,\beta]$.
\end{thm}


\section{The linking geometry and estimates for the critical level} 

In this section we verify that the functional $I$, defined in \eqref{Functional}, satisfies the geometrical properties of Theorem~\ref{linking}. We start by verifying the hypothesis $(I_0)$ and $(I_1)$.
{
\begin{lem}  There exist $ \rho , \ \sigma >0 $ such that
$ I ( z) \geq \sigma $, for all $ z \in \textit{S} := \partial
B_\rho  \cap \mathbb{H}^{+} $. Moreover, $I(z) \leq 0$ if $z \in \mathbb{H}^-$.
\end{lem}
\begin{proof} Assumption $(H2) $ implies that, for a given $ \epsilon_0 >0 $, there exists $ t_0 >0 $ such that
$$ f(t) \leq 2 \epsilon_0 t \;\;\;\mbox{ and }\;\;\; F(t) \leq \epsilon_0 t^2, \; \mbox{
for all } \; t \leq t_0.
$$
On the other hand, given $\alpha>\alpha_0$ and $ q > 2 $ there exists a positive
constant $C_1$ such that,
$$
F(t) \leq C_1 t^q (\exp \alpha t^2 - 1) , \mbox{ for all }  t
\geq t_0.$$
  Given $u \in H^{1}(\mathbb R)$ and define ${\Omega_{t_0}}:= \{ x \in
\mathbb{R} : 0 \leq u(x) \leq t_0\}$. Thus,
\begin{eqnarray*}
\int_{\mathbb{R}} P(x)F(u) \, \ud x & = & \int_{\Omega_{t_0}} P(x)F(u) \, \ud x
+\int_{{\mathbb{R} \setminus \Omega_{t_0}}} P(x)F(u) \, \ud x
\\
& \leq & \epsilon_0 \int_{\mathbb{R}}P(x)u^2 \,\ud x + C_1
\int_{{\mathbb{R} \setminus \Omega_{t_0}}}P(x)u^q (\exp \alpha u^2 - 1) \,\ud x.
\end{eqnarray*}
Now using the H\"{o}lder's inequality with $r^{-1}+s^{-1}=1$, we
obtain
\begin{equation} \label{ineq}
\begin{aligned}
\int_{\mathbb{R}} P(x)F(u)& \,\ud x  \leq \epsilon_0 \| u
\|^2_{L^2(\mathbb{R};P)}\\
&\quad+C_1 \| u \|^q_{L^{qr}(\mathbb{R};P)}
\left(\int_{{\mathbb{R} \setminus \Omega_{t_0}}}P(x)(\exp \alpha u^2 - 1)^{s}\,\ud x\right)^\frac1s.
\end{aligned}
\end{equation}
We use the following inequality to estimate the second term in \eqref{ineq}. There exists a constant $C_2=C_2(t_0)$ such that
\begin{equation*}
(\exp \alpha t^2 - 1 )^s \leq C_2(\exp \alpha s t^2 - 1 )
\;\;\mbox{for all}\;\;t \geq t_0\;
\end{equation*}
which implies 
\begin{equation*}
\int_{{\mathbb{R} \setminus \Omega_{t_0}}}P(x)(\exp (\alpha u^2) - 1)^{s}\,\ud x\leq C_2\int_{{\mathbb{R} \setminus \Omega_{t_0}}}P(x)(\exp (\alpha s u^2) - 1)\,\ud x.
\end{equation*}
From \eqref{vanishing}, we have 
\begin{align*}
\int_{{\mathbb{R} \setminus \Omega_{t_0}}}P(x)(\exp \alpha s u^2 - 1) \
\,\ud x
&\leq C \int_{\mathbb{R}} (\exp \alpha s u^2 - 1) \
\,\ud x
\end{align*}
 which applied to (\ref{ineq}) together with Lemma \ref{compact-a} gives
\[
\int_{\mathbb{R}} P(x)F(u) \, \ud x  \leq \epsilon_0 C \| u
\|^2_{1/2}+C_3\| u
\|^q_{1/2}\left(\int_{\mathbb{R}}(\exp \alpha s u^2
- 1)\,\ud x\right)^{\frac{1}{s}}.
\]
Analogously, possibly for different positive constants $C$ and $C_3$, we have
\[
\int_{\mathbb{R}} Q(x) G(u)\, \ud x  \leq \epsilon_0 C \| u
\|^2_{1/2}+C_3\| u
\|^q_{1/2}\left(\int_{\mathbb{R}}(\exp \alpha s u^2
- 1)\, \ud x\right)^{\frac{1}{s}}.
\]
The previous inequalities used in (\ref{Functional}) imply
\begin{eqnarray*}
I(u,u) & \geq & \|u\|_{1/2}^2 - C_1\epsilon_0 \| u
\|^2_{1/2} -  C_4\| u
\|^q_{1/2}\left(\int_{\mathbb{R}}(\exp \alpha s u^2
- 1) \,\ud x\right)^{\frac{1}{s}}.
\end{eqnarray*}
The Trudinger-Moser inequality as in Theorem  \ref{Prop:fractional TM} reads
$$
\begin{aligned}
\int_{\mathbb{R}}(\exp\alpha s u^2 - 1)\, \ud x = & \int_{\mathbb{R}}\left\{\exp\left[\alpha s \|u\|_{1/2}^2\left({\frac{u}{\|u\|_{1/2}}}\right)^2\right] - 1\right\}\, \ud x \leq & C
\end{aligned}
$$
whenever $\|u\|_{1/2}^2=\rho $ and $  \alpha s \rho  < \pi $ (we can choose $\rho$ sufficiently small for this to hold good). Hence
$$
I(u,u) \geq  \| u \|_{1/2}^2 - C_1 \epsilon_0 \| u \|_{1/2}^2 - C_5 \| u \|_{1/2}^q.
$$
Therefore, since $q>2$, we can find  $\sigma > 0$  sufficiently
small, such that $I(u,u) \geq \sigma > 0$ for $\|u\|_{1/2}= \rho$.
To complete the proof, one can see that  for any
$z=(u,-u) \in \mathbb{H}^-$, 
\[
I(u,-u)=- \|u\|_{1/2}^2 - \int_{\mathbb{R}}[ P(x)F(u) + Q(x)G(u) ] \, \ud x \leq
0,
\]
by using assumption $(H1)$.
\end{proof}
}

Now we obtain an upper bound for the minimax level,  and consequently $I$ satisfies $(I_2)$. The argument strongly relies on the Trudinger-Moser inequality and the growth condition $(H4)$. This estimate depends on an intricate reasoning involving the Moser sequence to be introduced. The idea is similar to that found in the celebrated work due to H.~Brezis and  L.~Nirenberg. However here we replace the Talenti's functions by the Moser's functions, defined as truncations and dilations of the fundamental solution: given $k \in \mathbb{N}$,
\[
m_{k} (x) =
\displaystyle{\frac{\overline{\omega}_k({x})}{\|\overline{\omega}_k\|_{1/2}}},
\]
where
\begin{equation}\label{modelfunction}
\overline{\omega}_k(x) = \frac{1}{\sqrt{\pi}}\left\{
\begin{array}{ll}
(\log k)^{1/2},  & \mid x \mid \leq 1/k, \\
\displaystyle{\frac{ \log \frac{1}{\mid x \mid} }{(\log k)^{1/2}}}, & 1/k \leq
\mid x \mid \leq 1, \\
0, & \mid x \mid \geq 1.
\end{array}
\right.
\end{equation}

\begin{lem}\label{estimate-1}  
Defining $\rho_k:=\frac{\log k}{ \pi}-m_k^2$, we have that there exists $C_o>0$ such that
\[
 0 \leq \rho_k \leq C_o \quad \mbox{for every} \quad |x| \leq \displaystyle{\frac{1}{k}}.
\]
\end{lem}
\begin{proof} 
Indeed, we have 
$$
\|\overline{\omega}_k\|_{1/2}^2=\int_{\R} | (-\Delta)^\frac14
\overline{\omega}_k({x})|^2 \, \ud x +\int_{B(0,1)}
|\overline{\omega}_k({x})|^2 \, \ud x = 1 + r_k.
$$ 
Here $ r_k \geq 0$ and direct calculations yield  $r_k=O\left(\frac{1}{\log k}\right)$ as $k\rightarrow \infty$ (see estimates $(2.3)-(2.5)$ in \cite{Takahasi}). Hence, $|r_k\log k|\leq C$ for $k>k_0$ leading to $\lim_{k\rightarrow\infty} r_k=0$. Therefore, if  $|x| \leq 1/k$, from the definition of $m_k$ and \eqref{modelfunction},  we get
\[
m_k^2:= \frac{\log k}{\pi {\|\overline{\omega}_k\|_{1/2}^2}}=\frac{\log k}{ \pi (1 + r_k)}=\frac{\log k}{\pi}
- \frac{(\log k) r_k}{\pi (1+ r_k)}.
\]
Hence
$$
\rho_k=\frac{(\log k) r_k}{ \pi (1+  r_k)}
$$ 
which together with the logarithmic decay estimate on $r_k$, as above, completes the proof.
\end{proof} 

\begin{prop}\label{progeometry}
	There exists $l_0 \in \mathbb{N} $ such that for all $l \geq l_0$ the corresponding Moser's function $m_l$ satisfies
	\[
	\sup_{\mathbb{R}^{+}(m_l,m_l)\oplus \mathbb{H}^- } I  < \frac{
		\pi}{\alpha_0}.
	\]
\end{prop}

\begin{proof} Suppose by contradiction that there exists a sequence $(l_k) \subset \mathbb{N} $ such that 
	$l_k \nearrow +\infty$ and for every $k \in \mathbb{N} $,
	\[
	\sup_{\mathbb{R}^{+}(m_{l_k},m_{l_k})\oplus \mathbb{H}^- } I \geq  \frac{\pi}{\alpha_0}.
	\]
	So, for every $ k\in \mathbb{N}, \; j\in \mathbb{N} $, there exists $u_{j,k} \in H^{1/2,2}(\mathbb R) $ and $\tau_{j,k} >0$ 
	such that
	\begin{equation}\label{UnB1}
	I(\eta_{j,k}) \geq  \frac{\pi}{\alpha_0}- \frac{1}{j},
	\end{equation}
	where
	\[
	\eta_{j,k} = \tau_{j,k}(m_{l_k},m_{l_k}) + (u_{j,k},-u_{j,k}) \in
	\mathbb{R}^{+}(m_{l_k},m_{l_k}) \oplus \mathbb{H}^- .
	\]
	Let $h:[0,\infty)\rightarrow \mathbb{R}$ be defined by $ h(t):=I(t\eta_{j,k})$. Since $h(0)=0$ and $\lim_{t\rightarrow+\infty} h(t)=-\infty$, there exists a maximum point $t_o\in (0,\infty)$ such that 
	\[
	h(t_o)=I(t_o \eta_{j,k}) \geq \frac{\pi}{\alpha_0}- \frac{1}{j}.
	\]
	Without loss of generality we may assume that $t_o=1$. Hence
	\begin{equation}\label{UnB2}
	I'(\eta_{j,k})\eta_{j,k}  =  0.
	\end{equation}
	Using \eqref{UnB1} and \eqref{UnB2}, we can write
	\begin{equation}\label{Itatiaia}
\begin{aligned}
	\tau^2_{j,k}  &\geq {\frac{\pi}{\alpha_0}}- \frac{1}{j} +
	\|u_{j,k}\|_{1/2}^2 +
	\int_{\mathbb{R}} P(x)F(\tau_{j,k}m_{l_k} + u_{j,k})\, \ud x\\
&\quad\quad\quad+\int_{\mathbb{R}} Q(x)G(\tau_{j,k}m_{l_k} -
	u_{j,k}) \, \ud x 
\end{aligned}
	\end{equation}
	and
	\begin{eqnarray}\label{Itatuba}
	&2\tau^2_{j,k} = 2\|u_{j,k}\|_{1/2}^2+
	\int_{\mathbb{R}} P(x) f(\tau_{j,k}m_{l_k} + u_{j,k})(\tau_{j,k}m_{l_k} + u_{j,k})
	\, \ud x\nonumber\\
&+ \int_{\mathbb{R}} Q(x)g(\tau_{j,k}m_{l_k} -
	u_{j,k})(\tau_{j,k}m_{l_k} -
	u_{j,k})\, \ud x.
	\end{eqnarray}
	From (\ref{Itatiaia}) it follows
	\begin{equation}\label{itatubal}
	\displaystyle{\frac{
			\pi}{\alpha_0}}+ s_{j,k}:=\tau_{j,k}^2 \geq
	\displaystyle{\frac{\pi}{\alpha_0}}- \frac{1}{j}.
	\end{equation}
	Therefore
	$\liminf_{j \to \infty}s_{j,k} \geq 0$.
	
	Consider $C_o>0$ given by Lemma~\ref{estimate-1} and take $\beta_0>0$ be such that 
	\begin{equation*}
	{\beta_0} > \frac{\pi}{2\alpha_0} \exp( \pi C_o).
	\end{equation*}
	
	By $(H4)$, there exists $R_0=R_0(\beta_0) > 0$ such that
	\begin{equation}\label{betalarge}
	{tf(t)}{ e^{ -\alpha_0 t^2}} \geq \beta_0
	\;\;\mbox{and}\;\; {tg(t)}{ e^{ -\alpha_0
			t^2}} \geq \beta_0,\;\; 
\mbox { for all } t \geq R_0.
	\end{equation}
	Therefore there exists $C>0$ such that $tf(t)\geq \beta_0e^{\alpha_0t^2}-C$ and $tg(t)\geq \beta_0e^{\alpha_0t^2}-C$ for $t\geq 0$.
	We also take $\bar{s}>0$ and $k_1\in {\mathbb N}$ such that for all $s\geq \bar{s}$ and $k\geq k_1$
	\begin{equation}\label{salvador2}
\begin{aligned}
	&2\left(\frac{\pi}{\alpha_0}+s\right) \exp(\pi C_o)\\& < \beta_0{l_k}  \exp\left(\alpha_0 {s}\left(\frac{\log {l_k}}{\pi} - \rho_{l_k}\right)\right)\int_{B_{{1}/{l_k}}(0)} \sqrt{P(x)Q(x)} \,\ud x. 
\end{aligned}
	\end{equation}

Next we use the inequality $e^x+e^y\geq 2e^{\frac{x+y}{2}}$ to get the following estimate 
\begin{align}\label{cabo1}
&\displaystyle\int_{{B_{{1}/{l_k}}(0)}} (P(x)  f(\tau_{j,k}m_{l_k} + u_{j,k})
	+Q(x)g(\tau_{j,k}m_{l_k} -
	u_{j,k}))(\tau_{j,k}m_{l_k} -
	u_{j,k})\, \ud x\nonumber\\
&\geq\beta_0\displaystyle\int_{{B_{{1}/{l_k}}(0)}} P(x) \exp(\alpha_0(\tau_{j,k}m_{l_k} + u_{j,k})^2)
	\,\ud x\nonumber\\
&\quad +\beta_0\displaystyle\int_{{B_{{1}/{l_k}}(0)}} Q(x) \exp(\alpha_0(\tau_{j,k}m_{l_k} -u_{j,k})^2)\,\ud x -O(1/l_k)\nonumber \\
&\geq 2\beta_0\displaystyle\int_{{B_{{1}/{l_k}}(0)}}\sqrt{ P(x) Q(x)} \exp(\alpha_0(\tau_{j,k}m_{l_k})^2)
	\,\ud x-O(1/l_k).
	\end{align}
  From (\ref{Itatuba}), \eqref{itatubal}, \eqref{betalarge}, \eqref{cabo1} and Lemma~\ref{estimate-1},
	\begin{align*}
		{\frac{
				\pi}{\alpha_0}}&+ s_{j,k}\\=&\tau_{j,k}^2
\geq  {\beta_0}
		\int_{B_{{1}/{l_k}}(0)} \sqrt{P(x)Q(x)} \exp\left(\alpha_0 \tau_{j,k}^2\left(\frac{\log l_{k}}{\pi} - \rho_{l_k}\right)\right){-O(1/l_k)}\\
		&\geq  \beta_0{l_k} \exp(- \pi \rho_{l_k}) \exp\left(\alpha_0 s_{j,k}\left(\frac{\log {l_k}}{\pi} - \rho_{l_k}\right)\right)\int_{B_{{1}/{l_k}}(0)}\sqrt{P(x)Q(x)} \,\ud x \\
&\quad{-O(1/l_k)}
	\end{align*}
	which is equivalent to 
	\begin{equation}\label{barcelona}
\begin{aligned}
	&\left(\frac{
		\pi}{\alpha_0}+ s_{j,k} \right) \exp( \pi \rho_{l_k})\\
&\geq
\beta_0{l_k}  \exp\left(\alpha_0 s_{j,k}\left(\frac{\log {l_k}}{\pi} - \rho_{l_k}\right)\right)\int_{B_{{1}/{l_k}}(0)} \sqrt{P(x)Q(x)} \,\ud x {-O(1/l_k).}
	\end{aligned}
\end{equation}
	Without loss of generality, we can assume that  for $k\geq k_1$ one has $\log {l_k} > 2 C_o \pi $. Then we claim that for every $k \geq k_1$, we have $s_{j,k} \leq \bar{s}$. Indeed, otherwise from Lemma~\ref{estimate-1},
	\begin{align*}
	&{2}\left(\frac{
		\pi}{\alpha_0}+ s_{j,k} \right) \exp( \pi C_o)  \\
	&\geq \beta_0{l_k}  \exp\left(\alpha_0 s_{j,k}\left(\frac{\log {l_k}}{\pi} - \rho_{l_k}\right)\right)\int_{B_{{1}/{l_k}}(0)} \sqrt{P(x)Q(x)} \,\ud x \\
	&\geq \beta_0{l_k}  \exp\left(\alpha_0 s_{j,k}C_o\right)\int_{B_{{1}/{l_k}}(0)} \sqrt{P(x)Q(x)} \,\ud x.
	\end{align*}
	However, this contradicts \eqref{salvador2}. The claim is proved.
	
	In view of the above claim and \eqref{itatubal}, we have 
	\[
	-\frac{1}{j} \leq s_{j,k} \leq \bar{s}\quad \mbox{for every }k \geq k_1 .\]
	Thus, taking a subsequence if necessary, we may suppose that 
	\[
	\lim_{j\rightarrow\infty} s_{j,k} = s_{o,k}\in [0,\bar{s}] \quad \mbox{for every} \quad  k \geq k_1.
	\]
	Consequently, from \eqref{barcelona}, for every $k\geq k_1$,
	\begin{eqnarray*}
	&\left(\frac{
		\pi}{\alpha_0}+ s_{o,k} \right) \exp ( \pi \rho_{l_k}) \\
&\geq \beta_0{l_k}  \exp\left(\alpha_0 s_{o,k}\left(\frac{\log {l_k}}{\pi} - \rho_{l_k}\right)\right)\int_{B_{{1}/{l_k}}(0)} \sqrt{P(x)Q(x)} \,\ud x{-O(1/l_k)} \;.
	\end{eqnarray*}
	Hence, since $s_{o,k}\geq 0$ and from Lemma~\ref{estimate-1}, we get 
	\begin{eqnarray} \label{bounda}
	&\left(\frac{
		\pi}{\alpha_0}+ s_{o,k} \right) \exp ( \pi C_o )  \exp( \alpha_0 s_{o,k} C_o )\nonumber\\
	&\geq  \beta_0{l_k}  \exp\left(\alpha_0 s_{o,k}\left(\frac{\log {l_k}}{\pi}\right)\right)\int_{B_{{1}/{l_k}}(0)} \sqrt{P(x)Q(x)} \,\ud x {-O(1/l_k)}\;.
	\end{eqnarray}
	Note that, from \eqref{bounda} and $s_{o,k} \in [0, \bar{s}]$ we have that $s_{o,k}\rightarrow 0$ as $k\rightarrow \infty$. Consequently, taking $k\rightarrow\infty$ in \eqref{bounda} and using \eqref{LDT}, we get
	\[
	\frac{ \pi }{2\alpha_0} \exp(\pi C_o) \geq {\beta_0} .
	\]
	However, this contradicts our choice of $\beta_0$. Thus Proposition~\ref{progeometry} is proved.
\end{proof}
\section{The $(PS)_c^*$ condition } 
{
In this section we shall verify that the functional $I$ satisfies the $(PS)_c^*$ condition with respect to an appropriate family of subspaces $(\mathbb{H}_n)_n$ of $\mathbb{H}$.
As a consequence of compact embedding results as in Lemma \ref{compact-a}, there exists an orthonormal basis $ \{\varphi_1, \varphi_2 \ldots  \}$ in $H^{1/2,2}(\mathbb R)$ of normalized  eigenfunctions associated to the eigenvalues of the compact operator $((-\Delta)^{-1/2}, P)$ in the weighted $L^2(\mathbb R;{P})$ space. Set,}
\begin{eqnarray*}
\mathbb{E}^{+}_n & = & \mathrm{span}\{(\varphi_i,\varphi_i): i=1,\ldots,n\} \\
\mathbb{E}^{-}_n & = & \mathrm{span}\{(\varphi_i,-\varphi_i): i=1,\ldots,n\} \\
\end{eqnarray*}
Considering $e=m_k$ given by Proposition~\ref{progeometry}, we define 
\begin{eqnarray*}
\mathbb{H}_n  & = & \mathbb{H}^{+}_n \oplus \mathbb{H}^{-}_n  \quad \mbox{and}\quad I_n=I|_{\mathbb{H}_n},\quad  \mbox{where}\\
\mathbb{H}^{+}_{n} & = & \Re(e,e) \oplus \mathbb{E}^{+}_n, \\
\mathbb{H}^{-}_{n} & = & \Re(e,-e) \oplus \mathbb{E}^{-}_n. 
\end{eqnarray*}
\begin{remark}
We observe that it is easy to see that if $z:=(u,v)\in \mathbb{H}_n $ then $(v,0)\in \mathbb{H}_n $ and $(0,u)\in \mathbb{H}_n $.  

\end{remark}

Next result establishes that $(PS)^*$ sequences are bounded.

\begin{prop}\label{propbounds} Given $c\in \mathbb{R}$, let  $\{z_n\}:=\{(u_n,v_n)\} \subset \mathbb{H} $ be a  $(PS)^*_c$ sequence for the functional $I$ with respect to the family $(\mathbb{H}_n)_{n}$ defined above. Then $\{z_n\}$ is bounded in $\mathbb{H}$. Furthermore, there exists $C_1>0$ such that
\begin{eqnarray*}
\int_{\mathbb{R}} P(x)f(u_n)u_n \, \ud x \leq C_1,\;\;\; \; \;\int_{\mathbb{R}} Q(x) g(v_n)v_n
\, \ud x \leq C_1,
\\
\int_{\mathbb{R}} P(x) F(u_n) \, \ud x \leq C_1,\;\;\;\; \;\;\;\;\int_{\mathbb{R}} Q(x) G(v_n)
\, \ud x \leq C_1.
\end{eqnarray*}
\end{prop}
\begin{proof} The sequence  $\{z_n\} \subset \mathbb{H} $  being a $(PS)_c^*$ sequence for $I$ satisfies the following  
\begin{equation*}
\textbf{(a)}\;\; I(u_n,v_n)\rightarrow c\;\;\;\; \text{ and } \;\;\;\;\textbf{(b)}\;\; \|I'_n(u_n,v_n)\|_{\mathbb{H}^*_n} \rightarrow 0, \text{ as } n \rightarrow\infty .
\end{equation*}
Taking  $ (\varphi, \psi)=(u_n, v_n)/\|(u_n,v_n)\| \in  \mathbb{H}_n  $ as testing functions according to $ \textbf{(b)} $, we have for some sequence $(\varepsilon_n)_{n\in \mathbb N}$ tending to $0$:
\begin{eqnarray*}
&\left| 2 \int_{\mathbb{R}} ((-\Delta)^\frac14 u_n (-\Delta)^\frac14 v_n +  u_n v_n)\,
\ud x - \int_{\mathbb{R}} (P(x) f(u_n)u_n +Q(x)g(v_n)v_n) \,\ud x \right| \\
&\leq \varepsilon_n \| (u_n, v_n ) \|
\end{eqnarray*}
which together with $ \textbf{(b)} $ and $(H3)$ imply {for some sequence $(\delta_n)_{n\in \mathbb N}$ tending to $0$ and $\theta>2$}
\begin{align*}
&\int_{\mathbb{R}} P(x)f(u_n)u_n \,\ud x + \int_{\mathbb{R}} Q(x)g(v_n)v_n \,\ud x \\
& \leq 2
\int_{\mathbb{R}} P(x) F(u_n)\, \ud x + \int_{\mathbb{R}} Q(x) G(v_n) \, \ud x + 2c + 2 \delta_n +
\varepsilon_n \|
(u_n, v_n ) \| \\
& \leq \frac{2}{\theta}\int_{\mathbb{R}} P(x) f(u_n)u_n \,\ud x+
\int_{\mathbb{R}} Q(x) g(v_n)v_n \, \ud x + 2c  + 2 \delta_n + \varepsilon_n \| (u_n, v_n )
\|.
\end{align*}
Thus
\begin{equation}\label{goiaba}
\int_{\mathbb{R}} P(x)f(u_n)u_n \,\ud x + \int_{\mathbb{R}} Q(x)g(v_n)v_n \,\ud x \leq  \frac{\theta}{\theta-2}(1+ 2
\delta_n + \varepsilon_n   \| (u_n, v_n ) \| ).
\end{equation}
Next taking  $ (\varphi, \psi)=(v_n, 0)/ \|v_n\|_{1/2} $ and $ (\varphi,
\psi)=(0, u_n)/\|u_n\|_{1/2} $ in $ \textbf{(b)} $ we have
\begin{eqnarray*}
\| v_n \|_{1/2} ^2 - {\varepsilon_n} \| v_n \|_{1/2} &  \leq &
\int_{\mathbb{R}}
P(x)f(u_n)v_n \, \ud x, \\
\| u_n \|_{1/2} ^2 - \varepsilon_n \| u_n \|_{1/2} &  \leq &
\int_{\mathbb{R}} Q(x) g(v_n)u_n \, \ud x .
\end{eqnarray*}
Setting $U_{n}=u_{n}/ \| u_n \|_{1/2} $ and  $ V_n = v_n / \| v_n \|_{1/2}
$ we have
\begin{eqnarray}
\| v_n \|_{1/2}   &  \leq & \int_{\mathbb{R}}
 P(x) f(u_n)V_n \; \ud x  + \varepsilon_n  \label{laranja},\\
\| u_n \|_{1/2}   &  \leq & \int_{\mathbb{R}}  Q(x)g(v_n)U_n \, \ud x +
\varepsilon_n . \nonumber
\end{eqnarray}
We now rely on the following {Young type} inequality  ( see \cite{Djairo}, Lemma~2.4):
\begin{equation}\label{FIDOORUF}
s\text{ }t\leq \left\{
\begin{array}{ll}
(e^{t^{2}}-1)+s(\log ^{+}s)^{1/2}, \; & \text{ for
all }t\geq 0 \text{ and }s\geq e^{1/4}, \\
(e^{t^{2}}-1)+\frac{1}{2}s^{2}, \; & \text{ for all } t \geq 0
\text{ and }0 \leq s\leq e^{1/4}.
\end{array}
\right.
\end{equation}
{Using the  critical growth of $f$ and $g$, for any $\epsilon>0$, there exists a constant $C_1=C_1(\epsilon)>0$ such that 
\begin{equation}\label{lago}
f(s)\leq C_1 e^{\alpha_0(1+\epsilon) s^2} \quad \mbox{and}\quad  g(s)\leq C_1 e^{(\alpha_0(1+\epsilon)s^2} \quad \mbox{for all}\quad  s \in \mathbb{R}.
\end{equation}}
Let 
$A_n=\left \{x \in \mathbb{R} : \frac{1}{C_1}f(u_n)(x) \geq e^{1/4} \right \}$, $B_n=\left \{x \in \mathbb{R} : \frac{1}{C_1}f(u_n)(x) \leq e^{1/4} \right \}$. Then 
\begin{equation}\label{morango}
\begin{aligned}
&\int_{\mathbb{R}} 
P(x) f(u_n)V_n \,\ud x  \leq  C_1 \int_{\mathbb{R}} P(x)(e^{ V_n^{2}}-1) \, \ud x+\\
&\int_{A_n} 
 P(x)f(u_n) \left [\log \frac{1}{C_1} f(u_n)\right ]^{1/2} \, \ud x 
+ \frac{1}{2C_1}\int_{B_n}
 P(x)\left ( f(u_n) \right )^2  \, \ud x. 
\end{aligned}
\end{equation}

By \eqref{lago}, the second integral on the right hand side yields
\begin{align*}
\int_{A_n} P(x)f(u_n) \left [\log \frac{1}{C_1} f(u_n)\right ]^{1/2} \,\ud x
&\leq
\int_{{A_n}}
P(x) f(u_n) \left[\log  e^{\alpha_0{(1+\epsilon)}u_n^2}\right]^{1/2} \, \ud x\\
&\leq{\sqrt{\alpha_0(1+\epsilon)}} \int_{\mathbb{R}} P(x) f(u_n) u_n\, \ud x.
\end{align*}
{
Note that by using condition $(H2)$, for some fixed  $s_o>0$ there exists a positive constant $C=C(s_o)$ we have $f(s)\leq C s$ for $0<s<s_o$. Therefore,
\[
f(s)^2 \leq C f(s)s \quad \mbox{for all} \quad 0 \leq s \leq s_o
\]
and
\[
f(s)^2 \leq \frac{C_1 e^{1/4}}{s_o} f(s)s \quad \mbox{for all}\quad  s \geq s_o \quad \mbox{when}\quad f(s)\leq C_1e^{1/4}
\]
}
which implies
\begin{align*}
\int_{B_n}
 P(x)\left [ f(u_n) \right ]^2  \,\ud x 
& =
\int_{\left \{x \in B_n:u_n(x) \in [0, s_o] \right \}}
 P(x)(f(u_n))^2  \, \ud x \\
& \quad+ \int_{\left \{x \in B_n : u_n(x) \geq s_o \right \}}
 P(x)(f(u_n))^2  \,\ud x \\&\leq
C\int_{\mathbb{R}} P(x)f(u_n)u_n \,\ud x.
\end{align*}
{
Now, in the light of Theorem \ref{Prop:fractional TM},  the first integral in the right hand side of  \eqref{morango} is bounded {\it i.e.}
\begin{align*}
C_1 \int_{\mathbb{R}} P(x)(e^{V_n^{2}}-1) \, \ud x&\leq C.
\end{align*}
}
Substituting the above estimates in (\ref{morango}), we obtain {for some positive constant $C$}
\[
\int_{\mathbb{R}}
P(x) f(u_n)V_n \, \ud x \leq  C \left(1+ \int_{ \mathbb{R} }  P(x) f(u_n) u_n \, \ud x \right).
\]
This estimate together with (\ref{goiaba})-(\ref{laranja})  imply
\begin{equation}\label{mandioca}
{\| v_n \|_{1/2}   \leq {C (1+ 2 \delta_n + \varepsilon_n.}
\| (u_n, v_n ) \|)} .
\end{equation}
Repeating the argument above it follows
\begin{equation}\label{macacheira}
{\| u_n \|_{1/2}   \leq {C(1+ + 2 \delta_n + \varepsilon_n.}
\| (u_n, v_n ) \|)} .
\end{equation}
Now considering the estimates (\ref{mandioca}) and (\ref{macacheira})
we finally obtain
\[
\| (u_n, v_n ) \| \leq C
\]
which completes the proof.
\end{proof}

\begin{prop}\label{Azul} Let $c\in (0,\pi/\alpha_0)$ and $\{z_n\}\subset \mathbb{H}$ be a $(PS)^*_c$ sequence for the functional $I$, with respect to the family $(\mathbb{H}_n)_n$ of \textcolor{red}{$\mathbb{H}$}, then $\{z_n\}$ possesses a subsequence which converges weakly in $\mathbb{H}$ to a critical point of $I$.
\end{prop}

\begin{proof}
By Proposition~\ref{propbounds}, $\{z_n\}$ is bounded sequence in $\mathbb{H}$.
Hence, invoking Lemma~\ref{compact-a} and the fact that $\mathbb{H}$ is a Hilbert space, $\{z_n\}$ has a subsequence (still denote by $\{z_n\}$) such that

\begin{eqnarray}
z_n & := & (u_n,v_n) \rightharpoonup z_o:=(u_o, v_o) \mbox{ in } \mathbb{H} , \label{Tambauzinho-1} \\
u_n \rightarrow u_o  & \mbox{ and } & v_n \rightarrow v_o \mbox{
in } L^q(\mathbb{R};P) \text{ and } L^q(\mathbb{R};Q), \nonumber
\; \forall \,{q \in[2,\infty)}, \\
u_n(x) \rightarrow u_o(x) & \mbox{ and } & v_n(x)
\rightarrow v_o(x) \mbox{ a.e. in } \mathbb{R}.\nonumber
\end{eqnarray}

We also note that, in view of Proposition~\ref{propbounds}, there exists $C>0$ such that, for every $n\in \mathbb{N}$,
\begin{equation}\label{saia-rodada}
\int_{\mathbb{R}} P(x) f(u_n)u_n \, \ud x \leq C,\quad \int_{\mathbb{R}} Q(x)g(v_n)v_n \,
\ud x \leq C.
\end{equation}
We claim that for any $ \phi \in C_0^\infty(\mathbb R)$ with $\mathrm {supp}\; \phi=K\subset \subset \mathbb R$, we have
\begin{equation}\label{Skank-1}
\int_{\mathbb R} P(x) f(u_n) \phi
\, \ud x \rightarrow \int_{\mathbb R} P(x) f(u_o) \phi
\, \ud x.
\end{equation}
{
Indeed, let $M>0$ and $\psi_M$ a smooth cut-off function such that
\begin{equation*}
\psi_M(t)=\left\{
\begin{array}{rl}
1 & \textrm{for}\;\; |t|\leq M\\
0& \textrm{for}\;\; |t|\geq M+1.
\end{array}
\right.
\end{equation*}
Then we have 
\begin{align*}
&\int_{\mathbb R} P(x) f(u_n) \phi
\, \ud x\\
&=\int_{\mathbb R} P(x)(1-\psi_M(u_n)) f(u_n) \phi \,\ud x+\int_{\mathbb R} P(x) \psi_M(u_n)f(u_n) \phi
\, \ud x\\
&=\int_{\mathbb R\cap \{u_n>M\}} P(x)(1-\psi_M(u_n)) f(u_n) \phi \,\ud x\\
&\quad +\int_{\mathbb R\cap  \{u_n<M+1\}} P(x) \psi_M(u_n)f(u_n) \phi
\, \ud x\\
&\leq \frac{\|\phi\|_\infty}{M}\int_{\mathbb R} P(x)f(u_n) u_n \,\ud x+\int_{\mathbb R\cap  \{u_n<M+1\}} P(x) \psi_M(u_n)f(u_n) \phi
\, \ud x.
\end{align*}
In the above inequality, using the Lebesgue dominated convergence theorem in the second integral and  \eqref{saia-rodada} in the first integral, we get for large $n$
\begin{align}\label{jrep}
\int_{\mathbb R} P(x) f(u_n) \phi
\, \ud x
&\leq \frac{C \|\phi\|_\infty}{M}+\int_{\mathbb R\cap  \{u_0\leq M+1\}} P(x) \psi_M(u_0)f(u_0) \phi
\, \ud x.
\end{align}
Now we estimate the second integral, keeping in mind the fact that $\psi_M(u_n)\to 1$ a.e. in $\mathbb R$ for $M$ large enough as
\begin{equation}\label{jrep2}
\begin{aligned}
\left|\int_{\mathbb R} \right.&P(x)f(u_0) \phi \,\ud x\left.-\int_{\mathbb R} P(x)\psi_M(u_0)f(u_0)\phi \,\ud x\right|\\&\leq \frac{\|\phi\|_\infty}{M}\int_{\mathbb R} P(x)f(u_0) u_0\,\ud x
=O(1/M).
\end{aligned}
\end{equation}
 Now for large $M>0$, we combine \eqref{jrep2} with \eqref{jrep} to get our required result. }
Similarly one can show that
\begin{equation}\label{Skank-2}
\int_{\mathbb R} Q(x) g(v_n) \psi
\, \ud x \rightarrow \int_{\mathbb R} Q(x) g(v_o) \psi
\, \ud x.
\end{equation}

Finally using \eqref{Tambauzinho-1}, \eqref{Skank-1} and \eqref{Skank-2} we conclude that
\begin{equation}\label{criticalp}
I'(u_o,v_o) (\phi, \psi) =0 \quad \mbox{for all} \quad (\phi, \psi)\in {C_0^\infty(\mathbb R)\times C_0^\infty(\mathbb R)}.
\end{equation}
{Using the fact that $\cup  \mathbb{H}_n$ is dense in $\mathbb{H}$  it follows that \eqref{criticalp}  holds for any  $(\phi,  \psi) \in \mathbb{H}$.}
 In other words, $z_o$ is a critical point of $I$. Proposition~\ref{Azul} is now proved. 
\end{proof}

\section{Proof of Theorem~\ref{thm1}} 

Arguing by contradiction, we suppose that the origin is the only critical point of the functional $I$. Note that the positivity of nontrivial weak solutions follows from the classical regularity theory and strong maximum principle for fractional Laplacian problems. 

We shall verify that, under this assumption, $I$ satisfies 
$(PS)^*_c$ for every $c\in (0, \pi/\alpha_0)$.

Let $\{z_n\}=\{(u_n,v_n)\} \in \mathbb{H}_n$ be such that 
\begin{eqnarray}
& I(z_n) \rightarrow c   \in (0, \pi/\alpha_0),  \label{bora-1}\\
& \| I_{n}'(z_n) \|_{\mathbb{H}^*_{n}}\rightarrow 0. &  \label{bora-2}
\end{eqnarray}
By Lemma~\ref{compact-a}, Propositions \ref{progeometry}, \ref{propbounds}
and the fact that the origin is the only possible critical point of $I$, we may assume that there exists $C>0$ such that
\begin{eqnarray}
\| z_n \| & \leq & C, \quad  \mbox{in} \quad \mathbb{H}  \label{Joaninha-1}\\
z_n=(u_n,v_n)  & \rightharpoonup & (0,0)  \quad  \mbox{weakly in} \quad \mathbb{H} , \nonumber \\
u_n & \rightarrow & 0  \quad  \mbox{strongly in } \quad L^q(\mathbb{R}; P), \nonumber\; \forall {q \in[2,\infty)},\\
v_n & \rightarrow & 0  \quad  \mbox{strongly in } \quad L^q(\mathbb{R}; Q), \nonumber
\; \forall {q \in [2,\infty)},
 \\
z_n(x) =(u_n(x),v_n(x)) &  \rightarrow & (0,0) \quad  \mbox{a. e. in } \quad  \mathbb{R},\nonumber\\
\int_{\mathbb{R}} P(x)f(u_n)u_n \, \ud x \leq C,  & \; & \int_{\mathbb{R}} Q(x)g(v_n)v_n \, \ud x
\leq C. \label{mulecada-1}
\end{eqnarray}
Since $f$ and $g$ belong to the class $\mathbb C_{\mathrm P}$ and $\mathbb C_{\mathrm Q}$ respectively, we get 
\begin{equation}\label{Muricoca}
\int_{\mathbb{R}} P(x)F(u_n) \, \ud x \rightarrow  0, \quad   \int_{\mathbb{R}} Q(x)G(v_n) \, \ud x \rightarrow  0 .
\end{equation}
From \eqref{Muricoca} and \eqref{bora-1}, we may find $\delta>0$ such that, for $n$ sufficiently large,
\[
\int_{\mathbb{R}} \left( (-\Delta)^\frac{1}{4}u_n (-\Delta)^\frac{1}{4} v_n + u_n v_n \right) \, \ud x
<  \frac{\pi}{\alpha_0}-2 \delta.
\]
This inequality, combined with  \eqref{bora-2} and \eqref{Joaninha-1}, implies, for $n$ sufficiently large,
\begin{equation}\label{MST-1}
\int_{\mathbb{R}} P(x) f(u_n)u_n\,\ud x + \int_{\mathbb{R}} Q(x)g(v_n)v_n \,\ud x \leq  \frac{\pi}{\alpha_0}-\delta .
\end{equation}

We claim that, given $\beta>\alpha_0$, for $n$ sufficiently large,
\begin{equation}\label{USP-1}
\|u_n\|_{1/2}+\|v_n\|_{1/2} \leq \left( \frac{\beta }{\alpha_0}\left( \frac{\pi }{\alpha_0}-\delta \right) \right)^{1/2}\;.
\end{equation}

In order to prove this claim we note that $\|u_n\|_{1/2}\not \rightarrow 0$. Indeed, assuming otherwise, from \eqref{Joaninha-1} we get that
\begin{equation*} 
\lim_{n\rightarrow\infty}\int_{\mathbb{R}} \left( (-\Delta)^\frac{1}{4} u_n (-\Delta)^\frac{1}{4} v_n + u_n v_n \right) \, \ud x=0
\end{equation*}
which together with \eqref{bora-1} and \eqref{Muricoca}, leads to a contradiction. Thus, we assume that $\|u_n\|_{1/2}\geq b>0$ {and $\|v_n\|_{1/2}\geq b>0$} for all $n$.
Now, set $\tilde u_n = (\frac{ \pi}{\alpha_0}- \delta)^{{1}/{2}}\frac{u_n}{\|u_n\|_{1/2}}$, $s= \frac{g(v_n)}{\sqrt \alpha_0}$ and $t= {\sqrt \alpha_0} \tilde u_n$. Then, using \eqref{bora-2} and \eqref{FIDOORUF} with $s$
and $t$ as above results into
\begin{equation}\label{3integral}
\begin{aligned}
&\left (\frac{\pi}{\alpha_0}- \delta \right )^{{1}/{2}}\|u_n\|_{1/2}\\&=\int_{\mathbb{R}}Q(x)g(v_n)\tilde u_n\, \ud x + {o_n(1)}\leq
\int_{\mathbb{R}}Q(x)(e^{\alpha_0 {\tilde u_n}^2}-1)\, \ud x \\ 
&\quad+
\int_{\left \{x \in \mathbb{R} : g(v_n)(x) \geq \sqrt{\alpha_0} e^{1/4} \right \}}
Q(x)\frac{g(v_n)}{\sqrt \alpha_0} \left [\log \left (\frac{g(v_n)}{\sqrt \alpha_0}\right ) \right ]^{1/2} \, \ud x \\
&\quad+\frac{1}{2}\int_{\left \{x \in \mathbb{R} : g(v_n)(x) \leq \sqrt{\alpha_0} e^{1/4} \right \}}
Q(x)\frac{(g(v_n))^2}{\alpha_0}\, \ud x+ {o_n(1)}.
\end{aligned}
\end{equation}

Let us evaluate each integral on the right hand side of \eqref{3integral}.
Taking $p>1$ and using the elementary inequalities for real numbers 
\[
\begin{aligned}
e^{\alpha_0 t^2}-1-\alpha_0 t^2 \leq \; & \alpha_0 t^2 \left(e^{\alpha_0 t^2}-1\right)\\
\quad 
\left( e^{\alpha_0 t^2}-1\right)^p \leq \; & C \left(e^{p\alpha_0 t^2}-1\right),
\quad \mbox{for} \quad t \geq 0,
\end{aligned}
\]
{
together with the H\"{o}lder inequality, we obtain 
\begin{equation*}
\begin{aligned}
&\int_{\mathbb{R}} Q(x)\left( e^{\alpha_0 \tilde{u}_n^2}-1-\alpha_0 \tilde{u_n}^2 \right) \, \ud x \\
&\leq \;  \alpha_0 \|\tilde{u}_n\|^2_{L^{2s}(\mathbb R; Q)} 
\left[\int_{\mathbb{R}} Q(x)\left( e^{\alpha_0 \tilde{u}_n^2}-1\right)^{s'} \,\ud x \right]^{1/{s'}}\\
&\leq  C\alpha_0 \|\tilde{u}_n\|^2_{L^{2s}(\mathbb R; Q)} 
\left[\int_{\mathbb{R}} \left( e^{\alpha_0s^\prime \tilde{u}_n^2}-1\right) \,\ud x \right]^{1/{s'}}
\end{aligned}
\end{equation*}
where $s'=s/(s-1)$. 
If $s'(\pi -\delta\alpha_0) < \pi$ we may apply Theorem \ref{Prop:fractional TM}, to find a constant $C>0$ such that}
\[
\int_{\mathbb{R}} Q(x) \left( e^{\alpha_0 \tilde{u}_n^2}-1\right) \; \ud x  \leq \alpha_0 \| \tilde{u}_n\|_{L^2(\mathbb R; Q)}^2+ C \| \tilde{u}_n\|_{L^{2s}(\mathbb R;Q)}^2.
\]
Now using compactness result from Lemma \ref{compact-a} in the above inequality, we get
\begin{equation}\label{muricoquinha}
\int_{\mathbb{R}} Q(x)\left( e^{\alpha_0 \tilde{u}_n^2}-1\right) \,\ud x  \rightarrow0, \quad \mbox{as}\quad n\rightarrow\infty.
\end{equation}
The second integral in \eqref{3integral} may be estimated as follows:
Considering $\beta_1>\alpha_0$, we may find {$C=C(\beta_1)>0$}, such that for any $t\geq 0$
\[
g(t)\leq C e^{\beta_1 t^2}.
\]
Hence
\begin{equation}\label{Bolinha}
\log\left(\frac{g(t)}{\sqrt{\alpha_0}} \right) \leq \log\left(\frac{C}{\sqrt{\alpha_0}} \right) + \beta_1 t^2.
\end{equation}
Using condition $(H2)$, there exists a positive constant $C$ independent of $n$ such that 
\begin{equation}\label{Bolao}
v_n(x) \geq C \quad \mbox{for all} \quad  x\in B_n
\end{equation}
where 
\[
B_n:=\left \{x \in \mathbb{R} :g(v_n)(x) \geq \sqrt{\alpha_0} e^{1/4} \right \}.
\]
Thus using \eqref{Bolinha}-\eqref{Bolao} we get
\begin{equation}\label{Rouxinol}
\begin{aligned}
&\int_{B_n}
Q(x)\frac{g(v_n)}{\sqrt \alpha_0} \left [\log \left (\frac{g(v_n)}{\sqrt \alpha_0}\right ) \right ]^{1/2} \, \ud x \\
&\leq
{\frac{1}{\sqrt{\alpha_0}}\sqrt{\log\left(\frac{C}{\sqrt{\alpha_0}}\right)}} \int_{B_n}
Q(x)g(v_n) \, \ud x
+
\frac{\sqrt{\beta_1}}{\sqrt{\alpha_0}} 
\int_{\mathbb{R}}
Q(x)g(v_n)v_n \, \ud x.
\end{aligned}
\end{equation}
We shall prove that
\begin{equation}\label{sabia}
\lim_{n\rightarrow \infty}\int_{B_n}
{Q(x)g(v_n) \, \ud x=0.}
\end{equation}
Indeed, given $M>1$, we consider $A_M^n:=\{x\in \mathbb{R}: v_n(x) \geq M\}$.  By $(H1)$ and $(H2)$, there exists $C>0$, independent of $n\in \mathbb{N}$, such that $g(v_n(x))v_n(x) \leq C v_n^2(x)$, for every $x\in B_n \setminus A_M^n$. Consequently, by \eqref{mulecada-1}, 
\[
\begin{aligned}
&\int_{B_n}
{Q(x)g(v_n) }\, \ud x = \;  \int_{B_n\setminus A_M^n}
{Q(x)g(v_n)} \, \ud x + \int_{B_n \cap A_M^n}
{Q(x)g(v_n)}\, \ud x \\
&\leq  \; \int_{\{B_n\setminus A^n_M\}\cap B^c_r(0)}Q(x)g(v_n)\, \ud x+\int_{\{B_n\setminus A^n_M\}\cap B_r(0)}Q(x)g(v_n)\, \ud x\\
&+
\frac{1}{M}\int_{\mathbb{R}}
Q(x)g(v_n)v_n \, \ud x.
\end{aligned}
\]
{Let $\epsilon>0$. For $L$ large enough and using \eqref{Bolao}, $(H2)$, the H\"older inequality together with $Q(x)<\epsilon$ for $|x|>L$, we have for some constant $\tilde C>0$
\begin{equation*}
 \int_{\{B_n\setminus A^n_M\}\cap B^c_L(0)}Q(x)g(v_n)\, \ud x\leq \epsilon C \int_{\{B_n\setminus A^n_M\}\cap B^c_L(0)}|v_n|^2\, \ud x \leq  \tilde C \epsilon.
\end{equation*}}
From compact embeddings, we have also for a fixed $L>0$
\begin{equation*}
 \int_{\{B_n\setminus A^n_M\}\cap B_L(0)}Q(x)g(v_n)\, \ud x=o_n(1).
\end{equation*}

{ Using above estimates} and noting that $M>1$ may be chosen properly large, the above inequality implies that \eqref{sabia} must hold.  

From \eqref{Rouxinol} and \eqref{sabia} we obtain 
\begin{equation}\label{Rouxinol-1}
\int_{B_n}
Q(x)\frac{g(v_n)}{\sqrt \alpha_0} \left [\log \left (\frac{g(v_n)}{\sqrt \alpha_0}\right ) \right ]^{1/2} \, \ud x 
\leq
\sqrt{\frac{\beta_1}{\alpha_0}} 
\int_{\mathbb{R}}
Q(x)g(v_n)v_n \, \ud x + \mathrm{o}_n(1).
\end{equation}
Next, we use $(H1)-(H2)$ and {Lemma \ref{compact-a}} to obtain 
\begin{equation}\label{Rouxinol-2}
\frac{1}{2}\int_{\left \{x \in \mathbb{R} : g(v_n)(x) \leq \sqrt{\alpha_0} e^{1/4} \right \}}
Q(x)\frac{(g(v_n))^2}{\alpha_0}\, \ud x 
\leq 
\frac{C}{2\alpha_0}\int_{\mathbb{R}}
Q(x)v_n^2\, \ud x 
= {o_n(1)}.
\end{equation}
{ Again using Lemma \ref{compact-a}}, \eqref{muricoquinha}, \eqref{Rouxinol-1}, \eqref{Rouxinol-2},  we may write 
\begin{equation}\label{coruja1}
\left (\frac{\pi}{\alpha_0}- \delta \right )^{{1}/{2}}\|u_n\|_{1/2}
\leq \mathrm{o}_n(1)+\sqrt{\frac{\beta_1}{\alpha_0}} 
\int_{\mathbb{R}}
Q(x) g(v_n)v_n \, \ud x. 
\end{equation}
Repeating the same argument for $v_n$, we also obtain
\begin{equation}\label{coruja2}
\left (\frac{ \pi}{\alpha_0}- \delta \right )^{{1}/{2}}\|v_n\|_{1/2}
\leq \mathrm{o}_n(1)+\sqrt{\frac{\beta_1}{\alpha_0}} 
\int_{\mathbb{R}}
P(x)f(u_n)u_n \, \ud x. 
\end{equation}
Choosing $\beta_1 \in (\alpha_0, \beta)$, the estimate \eqref{USP-1} is a direct consequence of \eqref{MST-1}, \eqref{coruja1}, \eqref{coruja2}.
The claim is proved.

Using \eqref{USP-1} we can choose $\beta$ close enough to $\alpha_0$ and
$\varepsilon>0$ such that 
$(\alpha_0 + \varepsilon) \max \{\|u_n\|_{1/2}^2, \|v_n\|_{1/2}^2\} < \pi$. Now using $(H1)-(H2)$ together with the exponential critical growth of $f$ and Theorem \ref{Prop:fractional TM} we obtain
\begin{align*}
&\int_{\mathbb{R}} P(x)f(u_n)u_n \,\ud x \leq \varepsilon \int_{\mathbb R}P(x)|u_n|^2 \,\ud x+ C_\varepsilon \int_{\mathbb{R}} P(x) u_n (e^{(\alpha_0+ \varepsilon) {u_n}^2}-1)\, \ud x\\
&\leq {\epsilon}\|u_n\|_{L^2(\mathbb R;P)}^2+C_\varepsilon \|u_n\|_{L^s(\mathbb R;P)}\left(\int_{\mathbb{R}} P(x) (e^{s'(\alpha_0+ \varepsilon) {u_n}^2}-1)\, \ud x\right)^{1/s'}
\end{align*}
where $1/s+1/s'=1$.  Now on choosing $s>2$ with $s'$ sufficiently close to $1$  such that $(\alpha_0 + \varepsilon) \max \{\|u_n\|_{1/2}^2, \|v_n\|_{1/2}^2\} s' <\pi$, we can do the similar calculations as before to control the integral in the second term of the above inequality. Hence using 
compactness from Lemma \ref{compact-a}, we obtain
\begin{equation} \label{unb1}
\int_{\mathbb{R}} P(x)f(u_n)u_n \, \ud x \to 0.
\end{equation}
Analogously, 
\begin{equation} \label{unb2}
\int_{\mathbb{R}}Q(x) g(v_n)v_n \, \ud x \to 0.
\end{equation}
Finally, using \eqref{unb1}, \eqref{unb2} and \eqref{bora-2}
we obtain 
$$
\int_{\mathbb{R}} \left( (-\Delta)^\frac14 u_n (-\Delta)^\frac14 v_n +  u_n v_n \right) \, \ud x \to 0
$$ 
which together with \eqref{Muricoca} gives a contradiction with \eqref{bora-1}.
The proof of the theorem is complete.
\section*{Acknowledgement}
Research was supported in part by INCTmat/MCT/Brazil, CNPq and CAPES/Brazil.

\section*{Appendix: Validation of Remark \ref{Val}}
In this appendix, we show the claim that under the assumptions $(H2)$ and $(H5)$,  $f$ 
(respectively $g$) belongs to the class $\mathbb C_{\mathrm P}$ (respectively $\mathbb C_{\mathrm Q}$) and that if $f(t)=O(t^2)$ then 
$u_n\rightharpoonup 0$ and $\{\int_\mathbb R P(x)f(u_n)u_n \,\ud x\}$ bounded imply that $\int_\mathbb R P(x)f(u_n)\,\ud x\to 0$.
\begin{proof}
Let $\{u_n\} \subset H^{1/2, 2}(\mathbb R)$ be a sequence such that 
\begin{equation}\label{Lbound}
u_n\rightharpoonup 0,\;\; \textrm{and}\;\; \int_\mathbb R P(x)f(u_n)u_n \,\ud x<C.
\end{equation}
For any given $\epsilon>0$, using $(H2)$ and $(H5)$, there exists $c_o=c_o(\epsilon)>0$ sufficiently small and $M>0$ sufficiently large such that
\[ 
F(t)\leq \epsilon |t|^2 \quad \text{for}\quad |t|<c_o \quad \text{ and  } \quad F(t)\leq \epsilon f(t)t \quad \textrm{ for }  t>M.
\]  
Hence from \eqref{Lbound}, we have 
\begin{equation}\label{wend}
\begin{aligned}
\int_{\mathbb{R}} P(x)F(u_n)\, \ud x&\leq \epsilon \int_{\{u_n\leq c_o\}}P(x)|u_n|^2\, \ud x+\int_{\{c_o\leq u_n\leq M\}}P(x)F(u_n)\, \ud x\\
&\quad+\epsilon\int_{\{u_n\geq M\}}P(x)f(u_n)u_n\, \ud x\\
&\leq \epsilon C+\int_{\{c_o\leq u_n\leq M\}}P(x)F(u_n)\, \ud x.
\end{aligned}
\end{equation}
Further for $L=L(\epsilon)>0$ large enough such that $P(x)<\epsilon$ for $x\in B^c_L(0)$, we have
\begin{equation}\label{est2}
\int_{\{c_o\leq u_n\leq M\}\cap B^c_L(0)}P(x)F(u_n)\, \ud x\leq \epsilon C,
\end{equation}
where $C$ is independent of $n$. Indeed, 
\[  
c_o^2| \{c_o\leq u_n\leq M\}|\leq \int_{\{c_o\leq u_n\leq M\}}|u_n|^2 \,\ud x \leq \int_{\mathbb R}|u_n|^2\,\ud x \leq C.
\]
Now by Lebesgue theorem, for a such fixed $L>0$, we have also 
\begin{equation}\label{est3}
\int_{\{c_o\leq u_n\leq M\}\cap [-L, L]}P(x)F(u_n)\, \ud x\to 0.
\end{equation}
Gathering \eqref{wend}-\eqref{est3},
we get 
\begin{equation*}
\int_{\mathbb{R}} P(x)F(u_n) \, \ud x \rightarrow 0. 
\end{equation*}
It finishes the proof of the first part of the claim. Next, we show the second statement of the claim. From $f(t)=O(t^2)$ and for $c_o$, $M>0$ respectively small and large enough, we have 
\begin{align*}
\int_\mathbb R P(x) f(u_n)\,\ud x&=\int_{\{u_n< c_o\}}P(x) f(u_n) \,\ud x+\int_{\{c_o\leq  u_n\leq M\}}P(x) f(u_n) \,\ud x\\
&\quad+ 1/M\int_{\{u_n> M\}} P(x)f(u_n)u_n
\,\ud x\\
 &\leq C \int_\mathbb R P(x) u_n^2 \,\ud x +\int_{\{c_o\leq  u_n\leq M\}}P(x) f(u_n) \,\ud x+C/M.
\end{align*}
Using Lemma \ref{compact-a} and estimating the second integral in the above inequality in a similar way as in \eqref{est2} and \eqref{est3}, we get the required result.
\end{proof}
\begin{remark}\label{Vanishing}
Let the function $P$ be defined as $P(x)=\frac{1}{(|x|+1)^\epsilon}$, for $\epsilon>0$ sufficiently small. Consider the sequence $\{u_n\}\subset H^{1/2,2}(\mathbb R)$ such that
\begin{eqnarray*}
u_n(x)=\displaystyle
\begin{cases}
\frac{1}{n^\alpha},& \mbox{ for }x\in [n,2n],\\
 \frac{(x-(n-1))}{n^\alpha},& \mbox{ for }x\in [n-1,n],\\
 \frac{((2n+1)-x)}{n^\alpha},& \mbox{ for }x\in [2n,2n+1],\\
0& \mbox{elsewhere}
\end{cases}
\end{eqnarray*}
 with $\alpha \in [1/2,1)$. Then, by straighforward calculations, we can prove $\{u_n\}$ is bounded in $H^{1/2,2}(\mathbb R)$. Furthermore, as $n\to\infty$
\begin{equation*}
\int_\R P(x)u_n^q\,\ud x\to 0
\end{equation*}
if and only if $q$ satisfies $\alpha q+\epsilon>1$. Therefore if $f$ is of $O(t^{1+\ell})$ near $0$ with $0<\ell\leq\frac{1-\alpha-\epsilon}{\alpha}$, we easily get that $u_n\rightharpoonup 0$ weakly in $H^{1/2,2}(\mathbb R)$ and 
\begin{equation*}
\int_\R P(x)f(u_n)u_n\,\ud x\to 0
\end{equation*}
as $n\to\infty$. However, $\int_\R P(x)f(u_n)\,\ud x\to 0$ is not verified.
\end{remark}

\end{document}